\documentclass[default]{sn-jnl}

\usepackage{graphicx}%
\usepackage{multirow}%
\usepackage{amsmath,amssymb,amsfonts}%
\usepackage{amsthm}%
\usepackage{mathrsfs}%
\usepackage[title]{appendix}%
\usepackage{xcolor}%
\usepackage{textcomp}%
\usepackage{manyfoot}%
\usepackage{booktabs}%
\usepackage{algorithm}%
\usepackage{algorithmicx}%
\usepackage{algpseudocode}%
\usepackage{listings}%

\newtheorem{theorem}{Theorem}
\newtheorem{proposition}[theorem]{Proposition}%
\newtheorem{example}{Example}%
\newtheorem{remark}{Remark}%

\newtheorem{definition}{Definition}%

\begin{document}

\title[Some spectral results for certain positive operators in Hilbert Spaces]{Some spectral results for certain positive operators in Hilbert Spaces}


\author*[1]{\fnm{Rashid } \sur{A. }}\email{rashid441188@gmail.com}
\author*[2]{\fnm{P Sam } \sur{Johnson }}


\affil*[1]{\orgdiv{Department of Mathematical and Computational Sciences}, \orgname{National Institute of Technology Karnataka (NITK)}, \orgaddress{\street{Surathkal}, \city{Mangaluru}, \postcode{575 025}, \state{Karnataka}, \country{India}}}
\affil*[2]{\orgdiv{Department of Mathematical and Computational Sciences}, \orgname{National Institute of Technology Karnataka (NITK)}, \orgaddress{\street{Surathkal}, \city{Mangaluru}, \postcode{575 025}, \state{Karnataka}, \country{India}}}




\abstract{This paper investigates spectral properties of certain classes of positive operators originated from different matrices appeared in linear complementarity problem. These positive operators play a crucial role in various areas of mathematics and its applications, including operator theory, functional analysis, and quantum mechanics. Understanding their spectral behavior is essential for analyzing the dynamics and stability of systems governed by such operators. P-matrix is one of the important types of matrices appearing in linear complementarity problems. In this research, with the help of spectral results we have given a factorization for P-matrix, as the product of two non-trivial P-matrices. We also focus on elucidating spectral properties such as eigenvalues, approximate eigenvalues and spectral values associated with certain positive operators.}




\maketitle
\section{Introduction}
Eigenvalues and eigenvectors, key concepts in linear algebra, have wide-ranging applications in both science and engineering. Certain properties of a matrix can be analyzed if the eigenvalues of the matrix are known to us. For example, if a matrix has no non-zero eigenvalue, then it is invertible. The spectrum is the infinite-dimensional analogue of the set of matrix
eigenvalues. Several core results of matrix theory are extended to linear operators on a Hilbert space, where
the proofs are typically quite different and one often needs additional assumptions on the operators.

\vspace{.1cm}

An $n \times n$ matrix $A \in \mathbb{R}^{n \times n}$ is called a P-matrix if all its principal minors are positive. The matrix $A \in \mathbb{R}^{n \times n}$ reverses the sign of the vector $x \in \mathbb{R}^{n}$ if $x_i(Ax)_i \leq 0$, for all $i \in \langle n \rangle$, where $x_i$ is the $i$-th component of the vector $x \in \mathbb{R}^{ n}$ and $ \langle n \rangle =\{1,2,3,\ldots,n\}$. Fiedler and Pták \cite{MR142565} have shown that $A$ is a P-matrix if and only if $A$ has sign non-reversal property, that is, if $x\neq 0 \in \mathbb{R}^n$, then there exists some index $j$ for which we have $x_j(Ax)_j >0$. Cottle et al. (\cite{MR3396730}) shown that given a real square matrix $A$, the linear complementarity problem $LCP(A,q)$ has a unique solution for each vector $q \in \mathbb{R}^n$ if and only if $A$ is a P-matrix. The P-matrices encompass such notable classes as the Hermitian positive deﬁnite matrices, the M-matrices, the totally positive matrices, the real diagonally dominant matrices with positive diagonal entries, and many more. The study of P-matrices has extended to the notion of P-operator to infinite-dimensional Banach spaces having a Schauder basis by  Kannan and Sivakuma \cite{MR3485836} and P-operator to infinite-dimensional Hilbert spaces relative to various orthonormal basis by  Rashid A. and P. Sam Johnson (\cite{rashid}). 

\vspace{.1cm}
 Spectral theory on Hilbert space forms the cornerstone of modern functional analysis, providing powerful tools for understanding the behavior of linear operators on infinite-dimensional spaces. In particular, the study of P-operators within this framework holds significant importance, with wide-ranging applications in diverse fields such as quantum mechanics, signal processing, and partial differential equations.
\vspace{.1cm}
The theory of P-Matrices (P-operators) explores transformations that preserve positivity, a concept crucial for modeling various physical phenomena and mathematical processes through linear complementarity problem. In this context, P-property is intimately linked with the notions of order, stability, and growth, offering profound insights into the underlying structure of the space under consideration.
\vspace{.1cm}

This paper serves as a primer to delve into the spectral theory of P-operators and sufficient operators on Hilbert space. We will embark on a journey to explore the spectral properties of these operators, uncovering their eigenvalues, eigenvectors, and spectral decompositions. Moreover, we will investigate the connections between P-property and the spectral characteristics of these operators, shedding light on their geometric, algebraic, and analytic properties.
\vspace{.1cm}
Through this exploration, we aim to provide a comprehensive understanding of spectral theory on Hilbert space for P-operators, equipping the reader with the necessary tools to tackle advanced topics in functional analysis and its applications.

\section{Preliminaries}

   \begin{theorem}(\cite{1})\label{t111}
Let $A \in M_n(\mathbb{C})$ be a P-matrix. Then the following statements hold.
\begin{enumerate}
    \item The matrix $A+D$ is a P-matrix, for any diagonal matrices $D$ with non-negative entries.
    
    \item  The inverse $ A^{-1}$ of $A$ is a P-matrix.
    
\end{enumerate}
\end{theorem}

\begin{theorem}(\cite{MR142565})
	Let $A\in \mathbb R^{n\times n}$. The following statements are equivalent :
	\begin{enumerate}
		\item $A$ is a $P$-matrix.
		\item The real eigenvalues of the principal submatrices of $A$ are positive. 
	\end{enumerate}
\end{theorem}

\begin{theorem}(\cite{MR142565})\label{spect1}
	Let $A$ be a real square matrix with all non-positive off-diagonal elements. The following statements are equivalent :
	\begin{enumerate}
		\item $A$ is a P-matrix.
		\item Each real eigenvalue of $A$ is positive.
		\item The real part of each eigenvalue of $A$ is positive. 
	\end{enumerate}
\end{theorem}

Let $\lambda_1, \lambda_2, \ldots, \lambda_n$ be $n$ real numbers or complex numbers having both conjugate pairs in the collection. The $k^{th}$ elementary symmetric function is defined by
$$\sigma_k(\lambda_1, \lambda_2, \ldots, \lambda_n) = \sum_{1\leq i_1<i_2<\cdots <i_k\leq n} \quad  \prod_{t=1}^k \lambda_{i_t}.$$ 
The elementary symmetric functions characterize the spectra of P-matrices and P$_0$-matrices as shown in the following result.
\begin{theorem} (\cite{HERSHKOWITZ198381})\label{spect-1}
	The set $\{\lambda_1, \lambda_2, \ldots, \lambda_n\}$ is a spectrum of some P-matrix iff $$\sigma_k(\lambda_1, \lambda_2, \ldots, \lambda_n)>0, \ k=1, 2, \ldots, n.$$
\end{theorem}

\noindent In the above result, if $\sigma_k(\lambda_1, \lambda_2, \ldots, \lambda_n)\geq 0, k=1, 2, \ldots, n$, it characterizes P$_0$-matrices. 

\begin{example}
	Let $A=\begin{pmatrix}
	-1&-1\\
	4&3
	\end{pmatrix}$ be a matrix whose spectrum is $\{1,1\}$. Though the symmetric functions $\sigma_k(\{1,1\})$ are positive, for  $k=1,2$, the matrix $A$ is not a P-matrix. However, by Theorem (\ref{spect-1}), the set $\{1,1\}$ is a spectrum of the P-matrix $\begin{pmatrix}
	1&0\\
	0&1
	\end{pmatrix}.$
\end{example}

If $\{\lambda_1, \lambda_2, \ldots, \lambda_n\}$ is the spectrum of a matrix having positive sums of principal minor, then it satisfies $$\sigma_k(\lambda_1, \lambda_2, \ldots, \lambda_n)>0, k=1, 2, \ldots, n.$$ Moreover, the converse is also true as shown in Theorem (\ref{spect-1}) : $\sigma_k(\lambda_1, \lambda_2, \ldots, \lambda_n)$ is positive for $k=1, 2, \ldots, n$ if and only if it is a spectrum of some P-matrix ($P_0$-matrix). Such a set $S$ is called a P-set ($P_0$-set).

\cite{Kellogg1972} proved that elements of a P-set cannot lie in the given wedge around the negative axis. More precisely he proved the following result.
\begin{theorem}(\cite{Kellogg1972})
	\begin{enumerate}
		\item If $\mathcal{K}=\{\lambda_1, \lambda_2, \ldots, \lambda_n\}$ is a P-set, then 
  \begin{equation}\label{bound-1}
      |arg \ \lambda_i|<\frac{n-1}{n}\pi, \quad i=1, 2, \ldots,n. 
  \end{equation}

		\item If $\mathcal{K}=\{\lambda_1, \lambda_2, \ldots, \lambda_n\}$, $\lambda_i\neq 0$, $i=1, 2, \ldots, n$, is a P$_0$ set, then $$|arg \ \lambda_i|\leq \frac{n-1}{n}\pi,  \quad i=1, 2, \ldots,n.$$
\noindent Equality holds in the above inequality if and only if \begin{eqnarray*}
			\sigma_k(\lambda_1, \lambda_2, \ldots, \lambda_n) & = & 0, \quad \quad k=1, 2, \ldots, n-1, \\
			\sigma_n(\lambda_1, \lambda_2, \ldots, \lambda_n) & > & 0. 
		\end{eqnarray*}
		
	\end{enumerate}
\end{theorem}

In \cite{MR0709362} it is showed that if the number of elements in $\{\lambda_1, \lambda_2, \ldots, \lambda_n\}$ in the right half plane, or in the left half plane, is given, then the bound of Equation (\ref{bound-1}) can be improved, namely, there exists $\alpha$ such that $$|arg \ \lambda_i|<\alpha < \frac{n-1}{n}\pi, \ i=1, 2, \ldots, n.$$
It is also proved that if $\mathcal{K}$ has exactly one element in the right half plane, then 
$$|arg \ \lambda_i|<\frac{2}{3}\pi, \ i=1, 2, \ldots, n.$$
\noindent The above inequality is independent of $n$. 
It was conjectured in \cite{MR0709362} that if a P-matrix $A$ has exactly $k$ eigenvalues in the left half plane ($k$ is an even integer), then 
$$|arg \ \lambda|<\frac{2k+1}{2k+2}\pi, \quad \text{for all } \lambda \in\sigma(A).$$
The following result has given a negative answer to the conjecture.
\begin{theorem}(\cite{HERSHKOWITZ198381}) Let $\mathbf{C}$ be a finite set of complex numbers, consisting of positive numbers and non-real conjugate pairs. Then one can obtain a set with all elementary symmetric functions positive by adding positive numbers to $\mathbf{C}$. 
\end{theorem}	
\noindent	Every P-matrix can have almost all of its eigenvalues in the left half of the complex plane.
	\begin{theorem}(\cite{HERSHKOWITZ198381})
		There exists a P-matrix all of those eigenvalues, except one when $n$ is even or two when $n$ is odd, have negative real parts.
	\end{theorem}

A matrix $A$ all of whose powers are P-matrices is denoted $A\in \mathcal{PM}$ (respectively, $A\in \mathcal{P_\text{0}M}$ if only $A^k\in \mathcal P_0$). If we denote the set of eigenvalues with multiple appearances corresponding to algebraic multiplicities of $A$ by $\sigma(A)$, it is then natural to raise the following question :
$$\text{if } A\in \mathcal PM, \text{ does } \lambda\in \sigma(A) \text{ imply }\lambda>0?$$
The above question has been answered affirmatively by \cite{HERSHKOWITZ198381} for $n\leq 4$ and it is not resolved for $n\geq 5$. 

\begin{definition}(\cite{MR1387274})\label{def_of_sufficient}
    A matrix $A \in \mathbb{R}^{n \times n}$ is called column sufficient (CSU) if $x_i(Ax)_i \leq 0$ for all $i \in \langle n \rangle$, implies $x_i(Ax)_i=0$ for all $i \in \langle n \rangle$. The matrix $A$ is ~~~~~~~~~~~~~~~~~~~~~~~~~~~~~~~~~~~~said to be row sufficient (RSU) if $A^T$ is column sufficient. A matrix
that is both row and column sufficient is simply called sufficient (SU). We denote the class of sufficient matrices by $\mathbf{S}$.
\end{definition}

\section{Factorization of P-matrices Using eigenvalues}

\begin{proposition}\label{specp1}
    Let $A\in \mathbb{R}^{n \times n}$ such that $I+A$ is invertible. We define 
    $$U(A)=(I+A)^{-1}(I-A).$$
    Then the following statements hold:
    \begin{enumerate}
        \item $A=U(U(A))=(I+U(A))^{-1}(I-U(A)).$
        \item $I+U(A)=2(I+A)^{-1}.$
        \item If $A$ is invertible, then $I-U(A)=2(I+A^{-1})^{-1}.$
    \end{enumerate}

\end{proposition}
\begin{proof}
\begin{enumerate}
    \item We have $(I+A)U(A)=I-A$. So, $U(A)+AU(A)=I-A$. Hence 
\begin{equation}\label{spece1}
    A(I+U(A))=I-U(A).
\end{equation}
It is noted that if $(I+U(A))x=0$, then $x=0$. Therefore $-1 \notin \sigma(U(A))$. Since $(I+U(A))^{-1}$ and $I+U(A)$ commute, by the equation (\ref{spece1}), we get that $A=U(U(A))$. 
\item We have 
\begin{eqnarray*}
    I+U(A)&=&I+(I+A)^{-1}(I-A)\\
    &=&(I+A)^{-1}[I+A+I-A]\\
    &=&2(I+A)^{-1}.
\end{eqnarray*}
\item We have 
\begin{eqnarray*}
    I-U(A)&=&I-(I+A)^{-1}(I-A)\\
    &=&(I+A)^{-1}[I+A-I+A]\\
    &=&2(I+A)^{-1}A.
\end{eqnarray*}
\noindent If $A$ is invertible, then we have 
\begin{eqnarray*}
    I-U(A)&=&2(I+A)^{-1}A\\
    &=&2(I+A)^{-1}(A^{-1})^{-1}\\
    &=&2[A^{-1}(I+A)]^{-1}\\
    &=&2(I+A^{-1})^{-1}.
\end{eqnarray*}
\end{enumerate}

\end{proof}
\begin{theorem}
    Let $A\in \mathbb{R}^{n \times n}$ be a P-matrix. Then $A$ is a product of P-matrices.
\end{theorem}
\begin{proof}
    By Theorem \ref{spect1} (1), each real eigenvalue of $A$ is positive. Hence $A$ has no negative real eigenvalues, So $(I+A)^{-1}$ and $(I-A)$ exist. Therefore $U(A)=(I+A)^{-1}(I-A)$ is well defined. By Proposition \ref{specp1}, we have
    \begin{eqnarray*}
        I+U(A)&=&2(I+A)^{-1}\\
        I-U(A)&=&2(I+A^{-1})^{-1}.
    \end{eqnarray*}
    By Theorem \ref{t111}, $I+U(A)$ and $I-U(A)$ are P-matrices. Thus $A=(I-U(A))^{-1}(I-U(A))$ is a factorization of $A$.
\end{proof}
\begin{definition}
    A matrix $A$ is called positive stable if all its eigenvalues have positive real parts. It is easy to see that a P-matrix may not be positive stable.
\end{definition}
\begin{proposition}\label{sm1}
    Let $A\in \mathbb{R}^{n \times n}$ be a P-matrix and let $D \in \mathbb{R}^{n \times n} $  be a positive diagonal matrix. Then $AD$ is positive stable.
\end{proposition}
\begin{proof}
    Since $A$ is a P-matrix, it has no negative real eigenvalues. As $D$ is a diagonal matrix with positive diagonal entries, every eigenvalue of $AD$ has positive real parts.
\end{proof}
The following result describes the factorization of a P-matrix into a product of positive stable matrices which are P-matrices as well. 
\begin{theorem}
    Let $A\in \mathbb{R}^{n \times n}$. If $A$ is a P-matrix, then there exists positive diagonal matrices $S$ and $T$ such that $S^{-1}AT$ is a product of positive stable P-matrices.
\end{theorem}
\begin{proof}
    By Theorem \ref{t111}, $I+U(A)$ and $I-U(A)$ are P-matrices. Thus by Proposition \ref{sm1}, $(I+U(A))S$ and $(I-U(A))T$ are  positive stable matrices. Thus,\\
    \begin{eqnarray*}
        S^{-1}AT&=&S^{-1}(I+U(A))^{-1}(I-U(A))T\\
        &=&[(I+U(A))S]^{-1}(I-U(A))T.
    \end{eqnarray*}
    This completes the proof.
\end{proof}

\section{Spectral results for certain positive operators}

We start this section by recalling definition of different spectrum of operators. We denote $\mathcal{B}(\mathcal{H})$ the collection of all bounded linear operators on Hilbert space $\mathcal{H}$.

\begin{definition}(\cite{MR1427262})
    Let $\mathcal{H}$ be a Hilbert space and $T \in \mathcal{B}(\mathcal{H})$. A scalar $k$ is called an eigenvalue of $T$ if there is a non-zero $x \in \mathcal{H}$ such that $T(x)=kx$. In this case, $x$ is called an eigenvector of $T$ corresponding to $k$.
    \par{} The set of all eigenvalue of $T$ is called the eigenspectrum of $T$. We denote eigenspectrum of $T$ by $\sigma_e(T)$.  Note that if $x$ is an eigenvector corresponding to an eigenvalue $k$ of $T$, then same hold for the unit vector $\frac{x}{||x||}$. Hence
    $$\sigma_e(T)=\{k \in \mathbb{K}: T(x)=kx \text{ for some } x \in \mathcal{H} \text{ with } ||x||=1 \}.$$
    \par{} Note that a scalar $k$ is an eigenvalue of $T$ if and only if the operator $T-kI$ is not injective. In that case, the closed subspace $ker(T-kI)$ is a non-zero subspace. It is called the eigenspace of $T$ corresponding to the eigenvalue $k$.
    \par{} A scalar $k$ is called an approximate eigenvalue of $T$ if the operator $T-kI$ is not bounded below, that is, for every $\beta > 0$, there is some $x\in \mathcal{H}$ with $||x||=1$ and $||T(x)-kx||<\beta$. The set of all approximate eigenvalues of $T$ constitutes the approximate eigenspectrum of $T$, we denote it by $\sigma_a(T)$. It is then clear that
    $$\sigma_a(T)=\{k \in \mathbb{K}:(T-kI)x_n  \longrightarrow 0 \text{ for some } (x_n) \in \mathcal{H} \text{ with } ||x_n||=1 \, \, \, \forall n \}$$
    \par{} A scalar $k$ is called spectral value of $T$ if the bounded operator $T-kI$ is not invertible in $\mathcal{B}(\mathcal{H})$. The set of all spectral values of $T$ is called the spectrum of $T$. We denote it by $\sigma(T)$.
    $$\sigma(T)=\{k \in \mathbb{K}:T-kI   \text{ is either not injective or not surjective }\}.$$
\end{definition}

\begin{theorem}\label{t1}(\cite{s1})
    (Spectral theorem for compact operators). Let $\mathcal{H}$ be an infinite dimen- sional Hilbert space, and let $T \in K(\mathcal{H})$ be a compact operator.
    \begin{enumerate}
        \item Except for the possible value $0$, the spectrum of $T$ is entirely point spectrum; in other words
$\sigma(T) - \{0\} = \sigma_p(T) - \{0\}$.
\item We have $0 \in \sigma(T)$, and $0 \in \sigma_p(T)$ if and only if $T$ is not injective.
\item The point spectrum outside of 0 is countable and has finite multiplicity: for each $\lambda \in  \sigma_p(T) - \{0\}$, we have
$dim(\lambda I- T)<+\infty$.
\item Assume $T$ is normal. Let $\mathcal{H}_0=ker(T)$, and $\mathcal{H}_1=ker(T)^\perp$ . Then $T$ maps $\mathcal{H}_0$ to $\mathcal{H}_0$ and $\mathcal{H}_1$ to $\mathcal{H}_1$; on $\mathcal{H}_1$, which is separable, there exists an orthonormal basis $(e_1,\ldots,e_n,\ldots)$ and $\lambda_n \in  \sigma_p(T) - \{0\}$ such that
$$lim_{n \rightarrow 0} \lambda_n=\{0\}$$
and
$T(e_n)=\lambda_ne_n$ for all $n>1$.

in  particular,  if $(f_i)_{i\in I}$ is  an  arbitrary  orthonormal  basis  of $\mathcal{H}_0$,  which  may  not  be separable, we have $$T(\sum_{i\in I} \alpha_if_i+\sum_{n\geq 1} \alpha_n e_n )=\sum_{n\geq 1} \lambda_n \alpha_n e_n$$ for  all  scalars $\alpha_n, \, \alpha_i \in \mathbb{C}$ for  which  the  vector  on  the  left-hand  side  lies  in $\mathcal{H}$,  and  the series on the right converges in $\mathcal{H}$.  This can be expressed also as 
\begin{equation}\label{speq1}
    T(v) =\sum_{n \geq 1}\lambda_n\langle v,e_n \rangle e_n.
\end{equation}

    \end{enumerate}
\end{theorem}

\begin{proposition}\label{p1}
    Let $T\in \mathcal{B}(\mathcal{H})$ be a P-operator relative to an orthonormal basis $\mathcal{B}=\{e_i\}_{i=1}^\infty$. Then any real eigen value of $T$ is positive.
\end{proposition}
\begin{proof}
    Let $\lambda \in \sigma_e(T)$. Then there exist $x \neq 0$ such that $T(x)=\lambda x$. As $T\in \mathcal{B}(\mathcal{H})$ be a P-operator relative to the orthonormal basis $\mathcal{B}$, there exists some index $j$ for which,
    \begin{eqnarray*}
       &\langle x, e_j \rangle\langle T x, e_j \rangle > 0,\\
       \implies &\langle x, e_j \rangle\langle \lambda x, e_j \rangle > 0,\\
        \implies &\langle x, e_j \rangle \lambda \langle  x, e_j \rangle  > 0,\\
       \implies &\lambda \langle x, e_j \rangle^2  > 0,\\
       \implies &\lambda > 0.
    \end{eqnarray*}
\end{proof}

\begin{proposition}\label{p2}
    Let $T\in \mathcal{B}(\mathcal{H})$ be a P-operator relative to an orthonormal basis $\mathcal{B}=\{e_i\}_{i=1}^\infty$. Then any real element in $\sigma_a(T)$ is positive.
\end{proposition}
\begin{proof}
    Let $\lambda \in \sigma_a(T)$ be real. Then there exist a sequence $\{x_n\} \in \mathcal{H}$ with $||x_n||=1$ and $Tx_n -\lambda x_n \longrightarrow 0 $ as $n \longrightarrow \infty$, that is, $(T -\lambda I) x_n \longrightarrow 0 $ as $n \longrightarrow \infty$, This implies that $Tx_n \longrightarrow \lambda x $ as $n \longrightarrow \infty$, where $x_n \longrightarrow x$ as $n \longrightarrow \infty$. As $T$ is a P-operator relative to the orthonormal basis $\mathcal{B}$ and $||x_n||=1$, there exist $j_n$ such that $\langle x_n, e_{j_n} \rangle$$\langle Tx_n, e_{j_n} \rangle >0 $, that is $\langle x, e_{j_n} \rangle \lambda \langle x, e_{j_n} \rangle >0 $ an $n \longrightarrow \infty$. Thus $\lambda \langle x, e_{j_n} \rangle ^2 > 0$, hence $\lambda > 0$.
\end{proof}

\begin{theorem}
     Let $T\in \mathcal{B}(\mathcal{H})$ be a normal operator and let $\mathcal{H}_0= ker(T)$ and $\mathcal{H}_1=ker(T)^{\perp}$. If $T'$, the restriction of $T$ to $\mathcal{H}_1$, is a P-operator relative to an orthonormal basis $\mathcal{B}'=\{e_i\}_{i=1}^\infty$, where $\mathcal{B}'$ is made of eigenvectors of $T$ with $T(e_i)=\lambda_i e_i$, Then there exists a positive operator $R \in K(\mathcal{H})$ such that $R^2 = T$, which is
    denoted $\sqrt{T'}$ or $(T')^{1/2}$. It is unique among positive bounded operators.
\end{theorem}
\begin{proof}
    We have $T(e_i)=\lambda_i e_i$, where $\lambda_i \in \sigma_e(T)-\{0\}$. As $T'$ is a P-operator relative to $\mathcal{B}'$, we have
      \begin{eqnarray*}
        \lambda_i &= &\lambda_i \langle e_i, e_i \rangle,\\
       & = &  \langle \lambda_i e_i, e_i \rangle,\\
       & = &  \langle T e_i, e_i \rangle,\\
       & = &  \langle  e_i, e_i \rangle \langle T e_i, e_i \rangle >0 \, \, \, \forall \, i.
    \end{eqnarray*}
    Now we define
    $$R(u_0+u_1)=\sum_{i \geq 1} \sqrt{\lambda_i} \alpha_i e_i.$$
    For any $u_0 \in \mathcal{H}_0$ and $u_1=\sum_{i \geq 1} \alpha_i e_i \in \mathcal{H}_1$. Then we can see that $R$ is a diagonal operator with coefficients $\sqrt{\lambda_i}$. Now by Theorem \ref{t1} we have $\sqrt{\lambda_i} \longrightarrow 0$ as $i \longrightarrow \infty$. This shows that $R$ is well defined compact operator. Thus for every $u\in \mathcal{H}_1$, we have

         \begin{eqnarray*}
        R^2(u) &= &R(\sum_{i \geq 1} \sqrt{\lambda_i} \alpha_i e_i),\\
       & = &  \sum_{i \geq 1} \lambda_i \alpha_i e_i,\\
       & = &  T'(u) \,\,\,\,\,  \forall u \in \mathcal{H}_1.\\
    \end{eqnarray*}
    Therefor $R^2=T'$.
    We next show the uniqueness. Let $R \in B(\mathcal{H})$ be such that $R^2=T'$. Then $RT'=RR^2=R^2R=T'R$. That is $R$ and $T'$ commute. It follows that any non-zero eigenvalue $\lambda_i$ of $T'$, $R$ induces the operator
    $$R_i: ker(T'-\lambda_iI)\longrightarrow ker(T'-\lambda_iI)$$
    given by $R_i=\sqrt{\lambda_i}I$ on the finite-dimensional $\lambda_i$ - eigenspace of $T'$, we can see that this $R_i$ are P-operators relative to the orthonormal basis of $ker(T'-\lambda_iI)$ and it satisfies $R_i^2=\lambda_ii I$. Also note that
     \begin{eqnarray*}
        \|R(u)\|^2 &= & \langle R(u), R(u) \rangle,\\
       & = &  \langle R^2(u), u \rangle,\\
       & = &  \langle T'(u), u \rangle.
    \end{eqnarray*}
    Thus \begin{eqnarray*}
        ker(R) &= & \{u \in \mathcal{H}: R(u)=0\},\\
       & = &  \{u \in \mathcal{H}: ||R(u)||^2=0\},\\
       & = &  \{u \in \mathcal{H}: \langle T'(u), u \rangle=0\}.
    \end{eqnarray*} 
Now by the expression \ref{speq1} of the Theorem \ref{t1}, we have    
      \begin{eqnarray*}
       T'(u) &= &\sum_{i \geq 1} \lambda_i \langle u, e_i \rangle e_i .
       \end{eqnarray*}
       Thus
       \begin{eqnarray*}
       & \langle T'(u), u \rangle &  =\langle \sum_{i \geq 1} \lambda_i \langle u, e_i \rangle e_i , u \rangle\\
       & & =\sum_{i \geq 1} \lambda_i |\langle u, e_i \rangle |^2 .
    \end{eqnarray*}
    In terms of orthonormal basis of eigenvectors $\{e_i\}_{i=1}^\infty$ and the positivity $\lambda_i >0 $ of eigenvalues, we have $\langle T'(u), u \rangle =0 $ if and only if $u$ is perpendicular to span of $\{e_i\}_{i=1}^\infty$, that is, by the construction $u \in ker(T')$. Thus the P-operator $R$ is uniquely determined on each eigenspace of $T'$ and $ker(T')$. By the spectral theorem, this implies that $R$ is unique.
\end{proof}

\noindent Next we define the concept of $k$th elementary symmetric function in infinite dimensional cases as follows:
\begin{definition}
    Let $\sigma_k(\lambda_1,\lambda_2,\lambda_3,\ldots)$ denote the $k$th symmetric function of the numbers $\lambda_1,\lambda_2,\lambda_3\ldots,$ given by
    $$\sigma_k(\lambda_1,\lambda_2,\lambda_3,\ldots)=\sum_{1 \leq i_1<i_2< \ldots < i_k }\lambda_{i_1}\cdot \lambda_{i_2}\cdot \ldots\cdot \lambda_{i_k}.$$
\end{definition}

\begin{theorem}
    The set $\{ \lambda_1,\lambda_2,\lambda_3,\ldots,\}$ is an eigen spectrum of some P-operator $T \in \mathcal{B}(\mathcal{H})$ relative to an orthonormal basis $\mathcal{B}=\{e_i\}_{i=1}^\infty $ if and only if \vspace{-0.2cm}
$$\sigma_k(\lambda_1,\lambda_2,\lambda_3,\ldots)>0, \,\,\, k=1,2,3, \ldots $$

\end{theorem}

\begin{proof}
    First Suppose that the set $\{ \lambda_1,\lambda_2,\lambda_3,\ldots,\}$ is an eigen spectrum of some P-operator relative to an orthonormal basis $\mathcal{B}=\{e_i\}_{i=1}^\infty $. As $\lambda_i \in \sigma_e(T)$, we have, for each $i$, $\lambda_i >0$ by Proposition \ref{p1}. As $\sigma_k(\lambda_1,\lambda_2,\lambda_3,\ldots)$ is the result of sums of the products of $\lambda_i$, we have $\sigma_k(\lambda_1,\lambda_2,\lambda_3,\ldots)>0$ for each $k$.
    \par{} Conversely, suppose that the set $\{ \lambda_1,\lambda_2,\lambda_3,\ldots,\}$ satisfies $\sigma_k(\lambda_1,\lambda_2,\lambda_3,\ldots)>0$ for each $k$. Let us define the operator $T \in \mathcal{B}(\mathcal{H})$ by
    $$T(x)=\sum_{i=1}^\infty \alpha \sigma_i \langle x, e_i \rangle e_i,$$ where $\alpha >0$ so small, so that $T$ is bounded and $\sigma_i$ denotes $\sigma_i=\sigma_i(\lambda_1,\lambda_2,\lambda_3,\ldots)$. Then $T$ is the required P-operator relative to the orthonormal basis $\mathcal{B}$.
\end{proof}

\begin{theorem}
    Let $T\in \mathcal{B}(\mathcal{H})$ be a P-operator relative to an orthonormal basis $\mathcal{B}=\{e_i\}_{i=1}^\infty$. Then $T$ satisfies

    $$\inf_{x \in \mathcal{H^+}} \sup_{e_i \in \mathcal{B}} \frac{\langle Tx, e_i \rangle}{\langle x, e_i \rangle} \geq 0$$
    where $\mathcal{H^+}=\{x\in \mathcal{H}:\langle x,e_i\rangle \neq 0 \text{ for all $i$}\}$. Moreover, if $T$ satisfies $Tu=\rho(T) u$ for some $u \in \mathcal{H^+}$, then

      $$\inf_{x\in \mathcal{H^+}} \sup_{e_i \in \mathcal{B}} \frac{\langle Tx, e_i \rangle}{\langle x, e_i \rangle} = \sup_{e_i \in \mathcal{B}} \inf_{x\in \mathcal{H^+}}\frac{\langle Tx, e_i \rangle}{\langle x, e_i \rangle} =\rho(T). $$

\end{theorem}
\begin{proof}
As $T$ is a P-operator relative to an orthonormal basis $\mathcal{B}=\{e_i\}_{i=1}^\infty$ there exist an index $j$ such that $\langle x, e_j\rangle \langle Tx, e_j \rangle >0$, thus $\frac{\langle x, e_j\rangle \langle Tx, e_j \rangle}{\langle x, e_j\rangle ^2}>0$, that is, $\frac{\langle Tx, e_j \rangle}{\langle x, e_j\rangle}>0$, so for any $x\in \mathcal{H^+}$, $\sup_{e_i\in \mathcal{H}} \frac{\langle Tx, e_j \rangle}{\langle x, e_j\rangle}>0$. Thus if we take infimum of these positive numbers we get that $\inf_{x \in \mathcal{H^+}} \sup_{e_i \in \mathcal{B}} \frac{\langle Tx, e_i \rangle}{\langle x, e_i \rangle} \geq 0$. Now

\begin{equation}\label{eqp41}
    \inf_{x \in \mathcal{H^+}} \sup_{e_i \in \mathcal{B}} \frac{\langle Tx, e_i \rangle}{\langle x, e_i \rangle} \leq \sup_{e_i \in \mathcal{B}} \frac{\langle Tu, e_i \rangle}{\langle u, e_i \rangle}=\sup_{e_i \in \mathcal{B}}\frac{\rho(T)\langle u, e_i \rangle}{\langle u, e_i \rangle}=\rho(T).
\end{equation}

\noindent Also we have,
\begin{equation}\label{eqp42}
    \sup_{e_i \in \mathcal{B}} \inf_{x\in \mathcal{ H^+}}\frac{\langle Tx, e_i \rangle}{\langle x, e_i \rangle}\geq \frac{\langle Tu, e_i \rangle}{\langle u, e_i \rangle}=\frac{\rho(T)\langle u, e_i \rangle}{\langle u, e_i \rangle}=\rho(T).
\end{equation}

Now by the well known result that for a function $\psi: X\times Y \longrightarrow X\times Y $, where $X$ and $Y$ are any two sets, we have
$$\inf_{x\in X} \sup_{y \in Y} \psi \geq \sup_{y \in Y} \inf_{x\in X}\psi$$ 

\noindent and from equations (\ref{eqp41}) and (\ref{eqp42}) we get

\begin{equation}
    \rho(T)\leq \sup_{e_i \in \mathcal{B}} \inf_{x\in \mathcal{ H^+}}\frac{\langle Tx, e_i \rangle}{\langle x, e_i \rangle} \leq \inf_{x \in \mathcal{H^+}} \sup_{e_i \in \mathcal{B}} \frac{\langle Tx, e_i \rangle}{\langle x, e_i \rangle} \leq \rho(T).
\end{equation} 
thus we get
$$\inf_{x\in \mathcal{H^+}} \sup_{e_i \in \mathcal{B}} \frac{\langle Tx, e_i \rangle}{\langle x, e_i \rangle} = \sup_{e_i \in \mathcal{B}} \inf_{x\in \mathcal{H^+}}\frac{\langle Tx, e_i \rangle}{\langle x, e_i \rangle} =\rho(T).$$
\end{proof}

\begin{theorem}
    Let $\mathcal{H}$ be a separable Hilbert space with an orthonormal basis $\{e_i\}_{i=1}^\infty$, $S,T \in \mathcal{B}(\mathcal{H})$ and $D$ be a diagonal operator such that $De_i=d_{ii}e_i$, where $0\leq d_{ii} \leq 1$. Then the following are true:
    \begin{enumerate}
        \item If $ST^{-1}$ is a P-operator relative to the orthonormal basis $\mathcal{B}$, then $0\notin \sigma_e(DT+(I-D)S)$.
        \item  If $S^{-1}T$ is a P-operator relative to the orthonormal basis $\mathcal{B}$, then $0\notin \sigma_e(TD+S(I-D))$.
    \end{enumerate}
\end{theorem}
\begin{proof}
\begin{enumerate}
    \item Let $ST^{-1} $ be a P-operator relative to $\mathcal{B}$. Suppose that $0\in \sigma_e(DT+(I-D)S)$. Then there exists some $0\neq x \in \mathcal{H}$ such that $0=(DT+(I-D)S)x=(D+(I-D)ST^{-1})Tx$. Set $ y=Tx$. Then $y\neq 0$, since $T$ is invertible. Now, $Dy+(I-D)ST^{-1}y=0$ implies $\langle De_i, e_i \rangle \langle y, e_i \rangle =- \langle (I-D)e_i, e_i \rangle   \langle ST^{-1}y, e_i\rangle $. If $\langle y, e_i \rangle \geq 0$, then $\langle (ST^{-1}y), e_i\rangle \leq 0$ so that $\langle y, e_i \rangle \langle (ST^{-1}y), e_i \rangle \leq 0$. If $\langle y, e_i \rangle \leq 0$, then $\langle (ST^{-1}y), e_i\rangle \geq 0$ so that $\langle y, e_i \rangle \langle (ST^{-1}y), e_i \rangle \leq 0$. That is, $\langle y, e_i \rangle \langle (ST^{-1}y), e_i \rangle \leq 0$ for all $i$, a contradiction to $ ST^{-1}$ being a P-operator relative to $\mathcal{B}$. Thus $0\notin \sigma_e(DT+(I-D)S)$.
    
    \item Let $S^{-1}T $ be a P-operator relative to $\mathcal{B}$. Suppose that $0\in \sigma_e(TD+S(I-D))$. Then there exists some $0\neq x \in \mathcal{H}$ such that $0=(TD+S(I-D))x=S(S^{-1}TD+(I-D))x$. As $S$ is invertible, we have $S^{-1}TDx+(I-D)x=0$. Now, $\langle (I-D)e_i, e_i \rangle \langle x, e_i \rangle =- \langle De_i, e_i \rangle   \langle S^{-1}Tx, e_i\rangle $. If $\langle x, e_i \rangle \geq 0$, then $\langle (S^{-1}Tx), e_i\rangle \leq 0$ so that $\langle x, e_i \rangle \langle (S^{-1}Tx), e_i \rangle \leq 0$. If $\langle x, e_i \rangle \leq 0$, then we get $\langle (S^{-1}Tx), e_i\rangle \geq 0$ so that $\langle x, e_i \rangle \langle (S^{-1}Tx), e_i \rangle \leq 0$. That is, $\langle x, e_i \rangle \langle (S^{-1}Tx), e_i \rangle \leq 0$ for all $i$, a contradiction to $ S^{-1}T$ being a P-operator relative to $\mathcal{B}$. Thus $0\notin \sigma_e(TD+S(I-D))$.
\end{enumerate}
    
\end{proof}
Next we discuss some spectral results for sufficient operators, which is operator form of sufficient matrices (see Definition \ref{def_of_sufficient}) in Hilbert space settings defined as follows.

\begin{definition}
Let $\mathcal{B}=\{e_i\}_{i=1}^\infty$ be an orthonormal basis of $\mathcal{H}$. Then $T \in \mathcal B(\mathcal H)$ is said to be: 

\begin{enumerate}
    \item C-sufficient relative to the given orthonormal basis $\mathcal{B}$ if for $x \in \mathcal{H}$, the inequalities $\langle x, e_i \rangle \langle Tx, e_i \rangle \leq 0$ for all $i$ imply that $\langle x, e_i \rangle \langle Tx, e_i \rangle = 0$ for all $i$.
    \item R-sufficient relative to the given orthonormal basis $\mathcal{B}$ if the adjoint of $T$ is a C-sufficient operator relative to $\mathcal{B}$.
    \item sufficient if it is both a C-sufficient and R-sufficient operator relative to the given orthonormal basis $\mathcal{B}$.
\end{enumerate}

\end{definition}

\begin{definition}(\cite{XU19931})
    Let $M$ be an $n \times n (n > 2)$ real matrix. The term $\lambda$-eigenvector denotes an eigenvector $u$ corresponding to the eigenvalue $\lambda$ of $M$, i.e., $Mu = \lambda u$. If
such an eigenvector $u = (u_1,\ldots, u_n)$~ has $u_i \neq 0$ for all $i = 1,\ldots, n$, then
$u $ is called a strictly nonzero eigenvector. On the other hand, if $u_i = 0$ for at
least one $i$, then $u$ is called a partly zero eigenvector.
\end{definition}

We define this concept of partly zero eigenvector in Hilbert space settings as follows:
\begin{definition}
    Let $\mathcal{H}$ be a separable Hilbert space with an orthonormal basis $\mathcal{B}=\{e_i\}_{i=1}^\infty$. Let $T\in \mathcal{B}(\mathcal{H})$. The term $\lambda$-eigenvector denotes an eigenvector $u$ corresponding to the eigenvalue $\lambda$ of $T$, i.e., $Tu = \lambda u$. If
such an eigenvector $u$ is such that $\langle u, e_i \rangle  \neq 0$ for all $i$, then
$u $ is called a strictly nonzero eigenvector. On the other hand, if $\langle u, e_i \rangle = 0$ for at
least one $i$, then $u$ is called a partly zero eigenvector.
\end{definition}
\begin{remark}
    Let $\mathcal{H}$ be a separable Hilbert space with an orthonormal basis $\mathcal{B}=\{e_i\}_{i=1}^\infty$ and $T\in \mathcal{B}(\mathcal{H})$. Let us denote $T_\alpha$ for $\alpha \subseteq \mathbb{N}$, the restriction of $T$ to $H'=span\{e_i: i\in \alpha \}$. Any vector in $H'$ we denote by $x_\alpha$.
\end{remark}
\begin{theorem}
    Let $\mathcal{H}$ be a separable Hilbert space with an orthonormal basis $\mathcal{B}=\{e_i\}_{i=1}^\infty$. Let $T\in \mathcal{B}(\mathcal{H})$. Then the following are equivalent.
    \begin{enumerate}
        \item $T$ is a C-sufficient operator relative to an orthonormal basis $\mathcal{B}$.
        \item For any index set $\alpha \subseteq \mathbb{N}$ and non-negative diagonal operator $D_\alpha \neq 0$, if $T_{\alpha}+D_\alpha$ is not invertible, then every zero eigenvector of $T_{\alpha}+D_\alpha$ is partly zero eigenvector.
    \end{enumerate}
\end{theorem}
\begin{proof} 
$(1) \Rightarrow (2)$: Suppose there exists an index set $\alpha \subseteq \mathbb{N}$, a nonnegative diagonal operator $D_\alpha \neq 0$, and a strictly nonzero $0$-eigenvector $x_\alpha$ such that $(T_\alpha+D_\alpha)x_\alpha=0$. Define a vector $y$ such that $y_\alpha = x_\alpha$ and
$y_{\overline{\alpha}} = 0$. Then $\langle y, e_i \rangle \langle Ty, e_i \rangle \leq 0$ for all $i$. It follows from the hypothesis that $\langle y, e_i \rangle \langle Ty, e_i \rangle = 0$ for all $i$. Since $x_\alpha$, is a strictly nonzero vector, we have $-Dx_\alpha = T_\alpha x_\alpha=0$, hence $D_\alpha=0$, which is a contradiction. Thus the zero eigenvector of $T_\alpha+D_\alpha$ is partly zero eigenvector.
\par   $(2) \Rightarrow (1)$: Suppose $\langle y, e_i \rangle \langle Ty, e_i \rangle \leq 0$ for any $i$. Let $\alpha=\{i:\langle x, e_i \rangle \neq 0\}$. If $T_\alpha x_\alpha \neq 0$, then there exists a nonnegative diagonal matrix $D \neq 0$ such
that $T_a\alpha= -Dx$. It follows that there exists a strictly nonzero vector x, satisfying $(T_\alpha+D)x_\alpha=0$, which is a contradiction. Hence $T_\alpha x_\alpha =0$, therefore, $\langle y, e_i \rangle \langle Ty, e_i \rangle \leq 0$ for any $i$. This completes the proof of the theorem. 
\end{proof}

\begin{theorem}
    Let $\mathcal{H}$ be a separable Hilbert space with an orthonormal basis $\{e_i\}_{i=1}^\infty$ and $T$ be a C-sufficient operator relative to $\mathcal{B}$. Then any eigenvector corresponding to a non-zero eigenvalue is not an element of $rev_{\mathcal{B}}(T)$.
\end{theorem}
\begin{proof}
    Let $\lambda \neq 0 \in \sigma_e(T)$. Then there exists a $x \neq 0 \in \mathcal{H}$ such that $Tx=\lambda x$. If $x \in rev_{\mathcal{B}}(T)$, then $\langle x, e_i \rangle \langle Tx, e_i \rangle \leq 0$ for all $i$. As $T$ is C-sufficient operator we have $\langle x, e_i \rangle \langle Tx, e_i \rangle = 0$ for all $i$, that is, $\langle x, e_i \rangle \langle \lambda x, e_i \rangle = 0$ for all $i$, Thus $\lambda \langle x, e_i \rangle^2 = 0$ for all $i$, this implies that $\langle x, e_i \rangle = 0$ for all $i$, thus $x=0$. This is a contradiction. Hence $x \notin rev_{\mathcal{B}}(T)$.
\end{proof}

{\bf Data Availability}
The authors confirms that we have not used any external data for this research work.

{\bf Declaration on Conflicts of Interest.}
We have no conflicts of interest to disclose. All authors declare that they have no conflicts of interest.

{\bf Acknowledgment.}  The first author thanks the National Institute of Technology Karnataka (NITK) Surathkal for giving financial assistance.

{\bf Funding.} This research received no specific funding. All authors declare that they have no conflicts of interest and competing interests.

{\bf Compliance with Ethical Standards.}
Not applicable.


\end{document}